\documentclass[11pt]{amsart}
\usepackage{amssymb}
\usepackage{dsfont}
\usepackage{color}

\usepackage[cp1250]{inputenc}
\usepackage{amssymb}
\usepackage[hmargin={24mm,24mm},vmargin={25mm,25mm}]{geometry}
\bibstyle{plain}
\newtheorem{theorem}{Theorem}[section]

\newtheorem{lemma}[theorem]{Lemma}
\newtheorem{problem}[theorem]{Problem}
\newtheorem{proposition}[theorem]{Proposition}
\newtheorem{corollary}[theorem]{Corollary}

\newtheorem{definition}[theorem]{Definition}
\newtheorem{remark}[theorem]{Remark}
\newcommand{\on}{\operatorname}

\newcommand{\G}{\mathcal G}
\newcommand{\F}{\mathcal F}
\newcommand{\K}{\mathcal K}
\newcommand{\N}{\mathbb N}

\newcommand{\R}{\mathbb R}

\newcommand{\mc}{\mathcal}
\newcommand{\mb}{\mathbf}

\newcommand{\m}{\mbox}
\newcommand{\pa}{\parallel}

\email {filip.strobin@p.lodz.pl}
\email{jswaczyna@wp.pl}


\begin{document}
\author{Filip Strobin}
\address{Institute of Mathematics, Lodz University of Technology,
W\'olcza\'nska 215, 93-005 {\L}\'od\'z, Poland}
\author{Jaroslaw Swaczyna}
\address{Institute of Mathematics, Lodz University of Technology,
W\'olcza\'nska 215, 93-005 {\L}\'od\'z, Poland}
\title{Connectedness of attractors of a certain family of IFSs}
\date{}

\begin{abstract}
Let $X$ be a Banach space and $f,g:X\rightarrow X$ be contractions. We investigate the set
$$
C_{f,g}:=\{w\in X:\m{ the attractor of IFS }\F_w=\{f,g+w\}\m{ is connected}\}.
$$
The motivation for our research comes from papers of Mihail and Miculescu, where it was shown that $C_{f,g}$ is a countable union of compact sets, provided $f,g$ are linear bounded operators with $\pa f\pa,\pa g\pa<1$ and such that $f$ is compact. Moreover, in the case when $X$ is finitely dimensional, such sets have been intensively investigated in the last years, especially when $f$ and $g$ are affine maps. As we will be mostly interested in infinite dimensional spaces, our results can be also viewed as a next step into extending of such studies into infinite dimensional setting. In particular, unlike in the finitely dimensional case, if $X$ has infinite dimension then $C_{f,g}$ is very small set (at least nowhere dense) provided $f,g$ satisfy some natural conditions.
\end{abstract}
\maketitle

\section{Introduction}
\subsection{{\color{black}Basics of the Hutchinson-Barnsley theory}}
$\;$\\
Let $X$ be a complete metric space and $\mc{F}$ be a finite family of continuous selfmaps of $X$. Then $\F$ generates a natural mapping $\F:\K(X)\rightarrow \K(X)$:
\begin{equation*}
\F(D):=\bigcup_{f\in\F}f(D),
\end{equation*}
where $\K(X)$ is the space of all nonempty and compact subsets of $X$, considered as a metric space with the Hausdorff metric $H$:
$$
H(D,G):=\max\left\{\sup_{x\in D}\left(\inf_{y\in G}d(x,y)\right),\sup_{y\in
G}\left(\inf_{x\in D}d(x,y)\right)\right\}=\inf\{r>0:D\subset G(r)%
\mbox{ and
}G\subset D(r)\},
$$
where $A(r):=\{x\in X:\exists_{y\in A}:d(x,y)<r\}$.\\
It turns out that if each $f\in\F$ is a Banach contraction (i.e., if the Lipschitz constant $Lip(f)<1$), then so is $\F:\K(X)\to\K(X)$. Hence it satisfies the thesis of the Banach fixed point theorem, which (in view of the fact that $\K(X)$ is complete provided $X$ is complete) means that there is a set $A_\mc{F}\in\K(X)$ such that  $\F(A_\mc{F})=A_\mc{F}$, and for any $D\in\mb{K}(X)$, the sequence of iterates $\F^{(n)}(D)$ converges (in Hausdorff metric) to $A_\mc{S}$. In fact, we have even more.\\
Let $g:X\rightarrow X$ be a mapping. If for some nondecreasing function $\varphi :[0,\infty
)\rightarrow \mathbb{R}$ with $\varphi^{(n)}(t)\to 0$ for $t\geq 0$, we have 
\begin{equation}
\forall _{x,y\in X}\mbox{ } d(g(x),g(y))\leq \varphi (d(x,y)),  \label{rk}
\end{equation}
then we say that $f$ is a \emph{Matkowski contraction}.\\
The following facts are known. The proof of the first one can be found in \cite{Mat} (cf. also \cite{GJ}), and of the second one in \cite{J}.
\begin{proposition}
Let $X$ be a complete metric space. If $g:X\rightarrow X$ is a Matkowski contraction, then it satisfies the thesis of the Banach fixed point theorem.
\end{proposition}
\begin{proposition}\label{p1}
Let $X$ be a metric space. If $\F$ is a finite family of Matkowski contractions of $X$, then $\F:\K(X)\to\K(X)$ is also a Matkowski contraction.
\end{proposition}
The above show that for a complete metric space $X$ and a finite family of Matkowski contractions $\mc{F}$, the function $\F:\K(X)\to\K(X)$ satisfies the thesis of the Banach fixed point theorem. This leads to the following definition:
\begin{definition}
\emph{Let $X$ be a complete metric space. Then every finite family of Matkowski contractions of $X$ is called an }iterated function system\emph{ (IFS is short).\\
If $\mc{F}$ is an IFS, then the fixed point of the mapping $\F:\K(X)\to\K(X)$ is called} a fractal\emph{ or }attractor in the sense of Hutchinson and Barnsley\emph{, and is denoted by $A_\mc{F}$}.
\end{definition}
Note that such a general notion of an IFS was considered in the literature (for example in \cite{Ha}, \cite{GJ}), but originally IFSs were defined as families of Banach contractions (\cite{H} and \cite{B}).

\subsection{{\color{black}Some notions of porosity}}
$\;$\\
Let $X$ be a Banach space and $M\subset X$.\\
We say that $M$ is \emph{c-porous}, if its convex hull $conv(M)$ is nowhere dense.\\
If $M$ is a countable union of c-porous sets, then we say that $M$ is \emph{$\sigma$-c-porous}.\\
We say that $M$ is \emph{strongly porous}, if
$$
\forall_{R>0}\;\forall_{x\in X}\;\forall_{\alpha\in(0,1)}\;\exists_{y\in X}\;\pa x-y\pa=R\;\mbox{ and }\;B(y,\alpha R)\cap M=\emptyset.
$$
If $M$ is a countable union of strongly porous sets, then we say that $M$ is \emph{$\sigma$-strongly porous}.\\
c-porosity and strong porosity were introduced and investigated in \cite{S}. In particular, by \cite[Proposition 2.8]{S}, we have that
\begin{proposition}\label{ppp}
If $M$ is c-porous, then it is strongly porous.
\end{proposition}
\begin{remark}\emph{
Note that strong porosity can be viewed as a stronger version of $R$-ball porosity -- cf. \cite{Z2}.
It turns out that c-porosity and strong porosity are one of the most restrictive notions of porosity. 
In particular, they are meager, but the converse is not true. Hence if we know that some set is not only meager, but also $\sigma$-strongly porous, then we know that it is even smaller. For more information on porosity we refer the reader to survey papers \cite{Z1} and \cite{Z2}.}
\end{remark}
The following Proposition is trivial (cf. \cite{Z2}).
\begin{proposition}\label{P1}
Let $X$ be an infinite-dimensional Banach space and $M\subset X$. If $M$ is relatively compact (i.e., its closure is compact), then it is c-porous, and, in particular, strongly porous.
\end{proposition}
\subsection{{\color{black}Main problem}}
$\;$\\
Let $X$ be a Banach space and $f,g:X\rightarrow X$ be Matkowski contractions. If $w\in X$, then we define $g_w$ by setting $g_w(x):=g(x)+w$, and denote the IFS $\mc{F}_w:=\{f,g_w\}$ (clearly, $g_w$ is a Matkowski contraction). Define
\begin{equation*}
C_{f,g}:=\{w\in X:A_{\mc{F}_w}\m{ is connected}\}.
\end{equation*}

{\color{black}

In the case of $X=\R^2$, sets $C_{f,g}$ are (using the terminology from \cite[Section 8]{B}) \emph{Mandelbrot sets} for the family of all IFSs $\F_w$, as each IFS $\F_w$ depends on the parameter $w\in\R^2$. Such sets (with parameter space of also higher, but finite dimension) were investigated for example in \cite{B} and \cite{SS}. In particular (see \cite[Lemma 2.3]{SS}), if $f,g:\R^n\to\R^n$ are affine maps so that $|\on{det}(f)|+|\on{det}(g)|\geq 1$, then $C_{f,g}=\R^n$, and in some special cases (see \cite[Lemma 2.4]{SS}), concrete parts of the set $C_{f,g}$ are fully determined by properties of $f$ and $g$.


A related direction of studies are connected with the important concept of the so called \emph{tailings}.
Let $A$ be an $n\times n$ expanding matrix (that is, with all eigenvalues greater than $1$) and $D=\{d_1,...,d_k\}\subset \R^n$ be finite set. Then there exists a unique nonempty and compact set $T\subset \R^n$ such that
\begin{equation*}
A(T)=\bigcup_{i=1}^k(T+d_i)
\end{equation*}
(if $T$ has positive Lebesgue measure, then it is called a {self affine tail}). 
 In the last two decades there has been intensive effort to study various aspects of sets $T$. In particular, there has been undertaken the question for which sets $D$ it is connected (see for example  \cite{HSV}, \cite{KL} and \cite{LW}). 
In \cite{LW} it is shown that there exists a certain norm on $\R^n$ making all maps $x\to A^{-1}x+A^{-1}d_i$ Banach contracting and hence the set $T$ can be viewed as the attractor of the IFS $\{A^{-1}+A^{-1}d_i:i=1,...,k\}$. 
Therefore the family of sets $C_{A^{-1}+s,A^{-1}}$, where $s\in\R^n$, codes the set of all two-element sets $D$ for which $T$ is connected (in the last Section we will also deal with such families directly). 

Miculescu and Mihail in \cite{MM1} and \cite{MM2} made a step towards the infinite dimensional case. They proved that if $f,g$ are linear and $f$ is compact (that is, the closure of the image of each bounded set is compact) then the set $C_{f,g}$ is a countable union of compact sets (see \cite[Theorem 4.1]{MM2} and \cite[Theorem 6]{MM1}). In the case of finite dimensional spaces this result says nothing, but if $X$ is infinite-dimensional, then by Proposition \ref{P1} we see that that $C_{f,g}$ is $\sigma$-c-porous and hence meager. In fact, as $C_{f,g}$ is always closed (see Theorem \ref{T1}), it is nowhere dense in that case. This means that for most parameters $w\in X$,  the set $A_{f,g+w}$ is not connected.

 We prove similar assertions under different assumptions. In particular, we replace the linearity conditions by some other ones.

}




{
\begin{remark}\emph{
In \cite{S1} we investigated a bit more general set than $C_{f,g}$. Namely, instead of contractions $f$ and $g$, we considered families of contractions $\F$ and $\G$. In this paper, since dealing with just two contractions, we obtain some further results, or (sort of) strengthenings of results from \cite{S1}. Note that, in fact, the first draft of the current paper had been created earlier than \cite{S1}.}
\end{remark}

\section{Main results}

The following result are particular versions of \cite[Theorem 3.1]{S1} and \cite[Lemma 2.3(i)]{S1}
\begin{theorem}\label{T1}
Let $X$ be a Banach space and $f,g:X\rightarrow X$ be Matkowski contractions. Then $C_{f,g}$ is closed.
\end{theorem}
\begin{lemma}\label{LL1}
Let $X$ be a Banach space and $f,g:X\to X$ be Matkowski contractions. Then the mapping $F(w)=A_{\F_w}$, $w\in X$, is uniformly continuous and, in particular, for any bounded $D\subset X$, the set $\bigcup_{w\in D}A_{\F_w}$ is bounded.
\end{lemma}
}

The first main result of this paper is the following one. In particular, part $(i)$ (and its proof) is a direct extension of mentioned results of Miculescu and Mihail. {Also, the case when $Lip(g)<\frac{1}{2}$ follows from \cite[Theorem 3.2]{S1}.}
\begin{theorem}\label{m1}
Let $X$ be a Banach space and $f,g$ be Matkowski contractions of $X$. Assume that $f$ is a compact operator (i.e., $f(D)$ is relatively compact for all bounded $D$; $f$ need not be linear), and $g$ satisfies one of the following conditions:\\
(i) {\color{black}$g$ is affine and $\on{Lip}(g)< 1$};\\
(ii) $Lip(g)\leq \frac{1}{2}$.\\
Then the set $C_{f,g}$ is a countable union of {\color{black}compact} sets. In particular, if $X$ is infinite dimensional, then $C_{f,g}$ is $\sigma$-c-porous.
\end{theorem}
Before we give the proof, let us state the corollary which is consequence of Theorem \ref{T1} and Theorem \ref{m1}.
\begin{corollary}\label{c1}
Assume that $X$ is infinite dimensional and $f,g$ are as in Theorem \ref{m1}. Then $C_{f,g}$ is nowhere dense.
\end{corollary}
\begin{proof}
By Theorems \ref{T1} and \ref{m1}, $C_{f,g}$ is closed and has empty interior. The result follows.
\end{proof}
\begin{remark}\label{R1}\emph{
Relation between being strongly porous, nowhere dense and countable union of relatively compact sets in infinite-dimensional spaces looks as follows: each strongly porous set is nowhere dense and there exist strongly porous sets which are not countable unions of relatively compact sets (consider proper closed infinite-dimensional linear subspaces). In separable spaces countable unions of relatively compact sets need not be nowhere dense (for example, sets which are  countable and dense), however in nonseparable spaces each countable union of relatively compact sets is at least nowhere dense. Hence both Theorem \ref{m1} and Corollary \ref{c1} are needed to get a full picture of the situation {(in \cite{S1} we omitted the discussion on porosity)}. }
\end{remark}
Now we present prove of Theorem \ref{m1}. We have to start with several lemmas. The first one is an extension of analogous observation from \cite{MM1}, {and is a special case of \cite[Lemma 3.3]{S1}. We give the proof for the sake of completeness.}
\begin{lemma}\label{lemat1}
Let $X$ be a complete metric space and $f,g$ be Matkowski contractions. For $\F=\{f,g\}$, we have
\begin{equation}
g(A_{\mc{F}})=\bigcup_{n\in\N}g^{(n)}(f(A_{\mc{F}}))\cup\{e\},\label{w}
\end{equation}
where $e$ is the fixed point of $g$.
\end{lemma}
\begin{proof}
At first, observe that $A_{\mc{F}}=f(A_{\mc{F}})\cup g(A_{\mc{F}})$. Hence $g(A_{\mc{F}})=g(f(A_{\mc{F}})) \cup g^{(2)}(A_{\mc{F}})$ and $$A_{\mc{F}}=f(A_{\mc{F}})\cup g(f(A_{\mc{F}})) \cup g^{(2)}(A_{\mc{F}}).$$
Proceeding inductively, we have that for any $n\in\N$, \begin{equation}g(A_{\mc{S}})=\bigcup_{i=0}^{n-1} g^{i}(f(A_{\mc{F}})) \cup g^{(n)}(A_{\mc{F}}).\label{w1}
\end{equation}
We are ready to prove (\ref{w}). At first observe that $e\in A_{\mc{F}}$. Indeed, since $e$ is a fixed point of $g$, we have $e\in \F(\{e\})$, which implies that for every $n\in\N$, $e\in \F^{(n)}(\{e\})$. Since $\F^{(n)}(\{e\})\rightarrow A_{\mc{F}}$, we have $e\in A_{\mc{F}}$ and, consequently, $e\in g(A_{\mc{F}})$. Hence, by (\ref{w1}), we have the containing $"\supset"$ in (\ref{w}).\\
Now we prove $"\subset"$. Let $x\in g(A_{\mc{F}})$ and assume that $x\notin g^n(f(A_{\mc{F}}))$ for every $n\in\N$. By (\ref{w1}), $x\in g^{(n)}(A_{\mc{F}})$ for every $n\in\N$. Since $g^{(n)}(A_{\mc{F}})\rightarrow \{e\}$ (because $\{e\}$ is an attractor of IFS $\{g\}$), we have $x=e$. This ends the proof. 
\end{proof}
\begin{lemma}\label{lemat2}
Let $X$ be a Banach space and $g:X\rightarrow X$. Assume that $g$ satisfies one of the following conditions:
\begin{itemize}
\item [(i)] $g$ is linear and $\pa g\pa<1$;
\item [(ii)] $Lip(g)\leq \frac{1}{2}$.
\end{itemize}
Then for any $n\in\N$ and $D\in\mb{K}(X)$, the set
$$
M_{g,n,D}:=\{w\in X:D\cap g^{(n)}_w(D)\neq \emptyset\}
$$
is compact, where $g^{(n)}_w$ is the $n$-th iterate of $g_w$ (recall that $g_w(x)=g(x)+w$).
\end{lemma}
\begin{proof}
Assume $(i)$ and let $ D \in K(X)$ and $n\in \N $. Then for every $x\in X$, we have
$$g^{(n)}_{w}(x)=g(g_w^{(n-1)}(x))+w=...=g^{(n)}(x)+g^{(n-1)}(w)+...+g(w)+w=g^{(n)}(x)+h(w),$$
for $h(w)=\sum^{n-1}_{i=0}{g^{(i)}(w)}$.
Clearly, $h$ is continuous and linear. Since $\pa g\pa<1$, it is also bijection and $h^{-1}$ is continuous (we have $h=(id-g^{(n)})(id-g)^{-1}$ and $(id-g^{(n)})$ and $(id-g)^{-1}$ are continuous and reversible; this follows for example from  \cite[Lemma 1, sec. 6, ch. VII]{K}).
Thus\\
$$M_{g,n,D}=\{w \in X: D \cap g^{(n)}_{w}(D) \neq \emptyset\}=\{w \in X:D\cap (g^{(n)}(D)+h(w))\neq \emptyset\}=$$ $$=\{w\in X: h(w)\in D-g^{(n)}(D)\}=h^{-1}(D-g^{(n)}(D)).$$
Hence $M_{g,n,D}$ is compact (as a continuous image of a compact set $D-g^{(n)}(D)$).\\
Now assume $(ii)$ and choose any $ D \in K(X)$ and $n\in N $. Let $(w_{k})\subset M_{g,n,D}$ be a sequence.
Then for any $k \in\N$, there exist $x_{k},y_{k}\in D$ such that $x_{k}=g^{(n)}_{w_{k}} (y_{k})=g(g^{(n-1)}_{w_k}(y_k))+w_k$. By choosing subsequences (we will not change index so the notation is more clear) we have $ x_{k} \rightarrow x $, $ y_{k} \rightarrow y $ for some $x,y \in X$ (recall that $D$ is compact).\\
If $n=1$, then each $w_k=x_k-g(y_k)$, hence $w_k\rightarrow x-g(y)$.\\
Assume that $n\geq 2$. We will show that $(w_{k})$ is Cauchy. In order to do so, let us take any $s,k \in \N$. We have
$$\pa w_{k}-w_{s} \pa=\pa x_{k} -g(g^{(n-1)}_{w_{k}} (y_{k}))-(x_{s}-g(g^{(n-1)}_{w_{s}} (y_{s}))) \pa \leq $$
$$\leq\pa x_{k} - x_{s} \pa + \pa g(g^{(n-1)}_{w_{k}}(y_{k})) -g(g^{(n-1)}_{w_{s}} (y_{s})) \pa \leq \pa x_{k} - x_{s} \pa + \frac{1}{2}\pa g^{(n-1)}_{w_{k}} (y_{k}) - g^{(n-1)}_{w_{s}} (y_{s}) \pa =$$ $$= \pa x_{k}-x_{s} \pa + \frac{1}{2} \pa g(g^{(n-2)}_{w_{k}} (y_{k})) + w_{k} -g(g^{(n-2)}_{w_{s}} (y_{s}))-w_{s} \pa \leq $$ $$\leq \pa x_{k}-x_{s} \pa + \frac{1}{2} \pa w_{k}-w_{s} \pa + \frac{1}{2} \pa g(g^{(n-2)}_{w_{k}}(y_{k})) - g(g^{(n-2)}_{w_{s}}(y_{s})) \pa \leq $$ $$\leq \pa x_{k}-x_{s} \pa + \frac{1}{2} \pa w_{k}-w_{s} \pa + \frac{1}{4} \pa g^{(n-2)}_{w_{k}}(y_{k}) - g^{n-2}_{w_{s}}(y_{s}) \pa \leq$$ $$\leq ... \leq \pa x_{k}-x_{s} \pa + \left( \frac{1}{2} + \frac{1}{4} + ... + \frac{1}{2^{(n-1)}} \right) \pa w_{k}-w_{s} \pa+\frac{1}{2^{n}} \pa y_{k}-y_{s} \pa.$$ 
Set $\alpha:=\left( \frac{1}{2} + \frac{1}{4} + ... + \frac{1}{2^{n-1}} \right)$. Then $\alpha<1$ and by the above calculations,
$$(1-\alpha) \pa w_{k} - w_{s} \pa \leq \pa x_{k} - x_{s} \pa + \frac{1}{2^{n}} \pa y_{k} - y_{s} \pa,$$ 
so
$$ \pa w_{k}-w_{s} \pa \leq \frac{1}{1-\alpha} \left(\pa x_{k} - x_{s} \pa + \frac{1}{2^{n}} \pa y_{k} - y_{s} \pa\right).$$
Since $(x_k)$ and $(y_k)$ are convergent, $(w_{k})$ is Cauchy, and therefore it is convergent.\\
All in all, every sequence of elements of $M_{g,n,D}$ has a convergent subsequence, which proves that it is relatively compact. It remains to show that $M_{g,n,D}$ is closed. Let $(w_k)\subset M_{g,n,D}$ and $w_k\rightarrow w$ for some $w\in X$. Then there are $x_k,y_k\in D$ such that $x_k=g^{(n)}_{w_k}(y_k)$. By choosing a proper subsequence, we can assume that $x_k\rightarrow x$ and $y_k\rightarrow y$ for some $x,y\in D$. By the continuity of $g$, we have that $g^{(n)}_{w_k}(y_k)\rightarrow g^{(n)}_w(y)$. On the other hand, $g^{(n)}_{w_k}(y_k)\rightarrow x$. Hence $x=g^{(n)}_w(y)$ and therefore $w\in M_{g,n,D}$
\end{proof}
We are ready to give proof of Theorem \ref{m1}.
\begin{proof}
{\color{black}At first observe that when dealing with (i), we can assume that $g$ is linear and $||g||<1$. Indeed, if $g$ is affine and $\on{Lip}(g)<1$, then $g=\tilde{g}+a$ for some linear $\tilde{g}$ with $||\tilde{g}||<1$ and some $a\in X$. Then $w\in C_{f,g}$ iff $A_{f,\tilde{g}+a+w}$ is connected iff $a+w\in C_{f,\tilde{g}}$ iff $w\in C_{f,\tilde{g}}-a$. Hence if $C_{f,\tilde{g}}$ is a countable union of compact sets, then so is $C_{f,g}$.}

For every $k\in \N$, put $D_{k}:=cl\left(f\left(\bigcup_{w\in B(0,k)}A_{\mathcal{S}_{w}}\right)\right)$ ($cl(\cdot)$ denotes the closure and $0$ is the zero element of $X$). Then each $D_k$ is compact. To see it, it is enough to show that $\bigcup_{w\in B(0,k)}A_{\mathcal{S}_{w}}$ is bounded. But this follows from Lemma \ref{LL1}.
By Lemma \ref{lemat1}, we have ($e_w$ is the fixed point of $g_w$)
$${\color{black}C_{f,g}}=\{w \in X:A_{\mathcal{S}_{w}}\mbox{ is connected}\}\subset \{w\in X:f(A_{\mathcal{S}_w})\cap g(A_{\mathcal{S}_w})\neq \emptyset\}=$$
$$ = \left\{ w \in X: f(A_{\mathcal{S}_{w}}) \cap \left( \bigcup_{n\in\N} g^{(n)}_{w}(f(A_{\mathcal{S}_{w}})) \cup \{e_{w} \}\right) \neq  \emptyset \right\}=$$
$$=\bigcup_{k\in\N}\left\{w\in B(0,k): f(A_{\mathcal{S}_{w}})\cap\left(\bigcup_{n\in\N}(g^{(n)}_{w}(f(A_{\mathcal{S}_{w}}))\cup\{e_{w}\}\right)\neq \emptyset \right\} \subset$$
$$ \subset\bigcup_{k\in\N} \left\{ w \in X:D_{k}\cap\left(\bigcup_{n\in\N}\left(g^{(n)}_{w}(D_{k})\right)\cup \{ e_{w} \} \right) \neq \emptyset \right\}=$$ $$=\bigcup_{k\in\N} \bigcup_{n\in\N}\{w \in X:D_{k}\cap g_{w}^{(n)}(D_{k})\neq \emptyset \} \cup \bigcup_{k\in\N}\{ w \in X: D_{k}\cap \{ e_{w} \} \neq \emptyset  \}= $$
$$
=\bigcup_{k\in\N}\bigcup_{n\in\N}M_{g,n,D_k}\cup\bigcup_{k\in\N}\{ w \in X: D_{k}\cap \{ e_{w} \} \neq \emptyset  \}=\bigcup_{k\in\N}\bigcup_{n\in\N}M_{g,n,D_k}\cup\bigcup_{k\in\N}\{ w \in X: e_w\in D_{k}\}.
$$ 
By Lemma \ref{lemat2}, each set $M_{g,n,D_k}$ is compact, hence to end the proof it is enough to show that for each compact set $D\subset X$, $\{w\in X:e_w\in D\}$ is relatively compact.\\
If $e_{w} \in D$, then due to $e_{w}=w+g(e_{w})$ (recall that $e_w$ is a fixed point of $g_w$), we have that $w=e_{w}-g(e_{w}) \in D-g(D)$. Hence $\{w\in X:e_w\in D\}\subset D - g(D)$ and the latter set is compact.
\end{proof}
The second main result of the paper is the following. {Note that the point (ii) follows from \cite[Theorem 3.6]{S1}, but point (i) needs some additional reasonings.}
\begin{theorem}\label{m2}
Let $X$ be an infinite dimensional Banach space and $f,g$ be Matkowski contractions such that both $f$ and $g$ are compact operators. Then
\begin{itemize}
\item[(i)] $C_{f,g}$ is strongly porous (in particular, nowhere dense).
\item[(ii)] $C_{f,g}$ is a countable union of compact sets, and in particular, $\sigma$-c-porous.
\end{itemize}
\end{theorem} 
Please compare thesis of this Theorem with Remark \ref{R1}.
We start with a lemma.
\begin{lemma}\label{lemat3}
Let $X,f,g$ be as in the assumptions of Theorem \ref{m2}. Then for every bounded set $D\subset X$, the set $E_D:=\left\{w\in D:A_{\mathcal{F}_w}\mbox{ is connected}\right\}$ is relatively compact.
\end{lemma}
\begin{proof}
Let $K:=cl\left(f\left(\bigcup_{w\in D}A_{\mathcal{F}_w}\right)\right)$ and $E:=cl\left(g\left(\bigcup_{w\in D}A_{\mathcal{F}_w}\right)\right)$. By Lemma \ref{LL1}, $E$ and $K$ are compact. We have
$$
E_D\subset\left\{ w \in D: f(A_{\mathcal{F}_{w}}) \cap g_w(A_{\mathcal{F}_{w}}) \neq \emptyset \right\} \subset \left\{w \in D : K \cap (E+w) \neq \emptyset \right\}=
$$
$$
=\left\{ w \in D: w \in K-E \right\} = D\cap (K-E).
$$
\end{proof}
\begin{proof}(of Theorem \ref{m2})\\
We first prove $(i)$.\\
Let $R>0$, $x\in X$ and $\alpha\in(0,1)$. Put $D:=B(x,2R)$ and let $E_D$ be as in Lemma \ref{lemat3}. By this Lemma, $E_D$ is relatively compact, hence c-porous. By Proposition \ref{ppp}, there is $y\in X$ such that $\pa y-x\pa=R$ and $B(y,\alpha R)\cap E_D=\emptyset$. It is enough to show that $B(y,\alpha R)\cap C_{f,g}=\emptyset$. Let $z\in B(y,\alpha R)$. Since $z\notin E_D$, we have that either $z\notin B(x,2R)$ or $A_{\mathcal{F}_z}$ is not connected. But 
$$\pa x-z\pa\leq\pa x-y\pa+\pa y-z\pa<R+R=2R,$$
so $z \in B(x,2R)$ and $A_{\mathcal{F}_z}$ is not connected. This means that $z\notin C_{f,g}$.\\
Now we prove $(ii)$.\\
For every $k\in\N$, let $D_k:=B(0,k)$. Then $C_{f,g}=\bigcup_{k\in\N}E_{D_k}$, where $E_{D_k}$ is defined as in Lemma \ref{lemat3}. Thus the result follows from this Lemma and a fact that $C_{f,g}$ is closed.
\end{proof}
{\color{black}
\begin{remark}\emph{
The proof of Theorem \ref{m2} gives a simple criterion for $w\in X$ for which $A_{\F_w}$ is not connected in the case when $f,g$ are linear operators. At first observe that in this situation, for every $t\in\R$ and $w\in X$, we have $tA_{\F_w}=A_{\F_{tw}}$. Indeed, we have
\begin{equation*}
f(tA_{\F_w})\cup (g(tA_{\F_w})+tw)=t(f(A_{\F_w})\cup (g(A_{\F_w})+w))=tA_{\F_w}.
\end{equation*}
Hence we can deal only with $w$ belonging to the unit sphere $S_X$. Now if $\alpha:=\max\{||f||,||g||\}$, then for every $w\in S_X$, the attractor $A_{\F_w}$ is a subset of the closed ball $\overline{B}\left(0,M\right)$, for $M:=\frac{1}{1-\alpha}$. Indeed, as $f$ is linear, we see that $f\left(\overline{B}\left(0,M\right)\right) \subset \overline{B}(\left(0,M\right)$, and if $x\in \overline{B}(\left(0,M\right)$, then $$||g(x)+w||\leq ||g||||x||+1\leq \alpha M+1\leq M.$$
Hence $f(\overline{B}(0,M))\cup (g(\overline{B}(0,M))+w)\subset \overline{B}(0,M)$ and thus $A_{\F_{w}}\subset \overline{B}(0,M)$. Then, setting $K:=f(\overline{B}(0,M))$ and $E:=g(\overline{B}(0,M))$, we have that
$$
(C_{f,g}\cap S_X)\subset (K-E).
$$
(in particular, if $f,g$ are compact, then $C_{f,g}\cap S_X$ is compact).}
\end{remark}
}

\begin{problem}\emph{
It is natural to ask, how far can we go with weakening the assumptions on $f$ and $g$. Note that in view of Theorem \ref{T1}, if $C_{f,q}$ has an empty interior, then it is automatically nowhere dense. Our conjecture is that if $X$ is infinite dimensional, $f,g$ are Matkowski contractions such that $f$ is a compact operator, then $C_{f,g}$ is nowhere dense (\cite[Theorem 3.1]{MM2} suggests that it might be impossible to omit the assumption that $f$ is compact).}
\end{problem}

\begin{corollary}\label{t6}
Let $X$ be an infinite dimensional separable Banach space and $f,g$ be Matkowski contractions with $f$ compact and such that one of the following conditions hold:
\begin{itemize}
\item[(i)]  $Lip(g)\leq \frac{1}{2}$;
\item[(ii)] $g$ is compact;
\item[(iii)] $g$ is {\color{black}affine} and $||g||<1$.
\end{itemize}
Then the set
$$
W_{f,g}:=\{(s,w)\in X\times X:A_{\{f_s,g_w\}} \mbox{is connected}\}
$$
is closed and nowhere dense. 
\end{corollary}
\begin{proof}
The fact that $W_{f,g}$ is closed can be proved in a similar way as in the case of set $W_{f,g}$ (in fact, the closedness was proved in \cite[Theorem 4,2]{S1}). Since every $(W_{f,g})_s$ has an empty interior, also $W_{f,g}$ must have empty interior.
\end{proof}


\end{document}